\documentclass{amsart}

\usepackage{amssymb}

\numberwithin{equation}{section}

\newcommand{\R}{\mathbb{R}}
\newcommand{\Z}{\mathbb{Z}}

\newcommand{\be}{\begin{enumerate}}
\newcommand{\ee}{\end{enumerate}}
\newcommand{\ba}{\begin{eqnarray*}}
\newcommand{\ea}{\end{eqnarray*}}

\newtheorem{thm}{Theorem}[section]
\newtheorem{cor}[thm]{Corollary}
\newtheorem{prop}[thm]{Proposition}

\theoremstyle{definition}
\newtheorem{dfn}[thm]{Definition}
\newtheorem{example}[thm]{Example}
\newtheorem{remark}[thm]{Remark}

\begin{document}

\title{Frobenius Polytopes}
\author[J. Collins]{John Collins}
\author[D. Perkinson]{David Perkinson}
\address{Reed College, Portland Oregon}
%\email{}
%\keywords{}
%\subjclass{}
\date{1/21/04}

\begin{abstract}
A real representation of a finite group naturally determines a
polytope, generalizing the well-known Birkhoff polytope.  This paper
determines the structure of the polytope corresponding to
the natural permutation representation of a general Frobenius group.
\end{abstract}

\thanks{The authors would like to thank Rao Potluri for many useful
insights.  The second author would like to thank Reed College students
Judy Ridenour and Hana Steinkamp.}

\maketitle
\section{Introduction} 
The collection of $n\times n$ matrices over the real numbers is the
$n^2$-dimensional Euclidean space $\R^{n\times n}$.  Given a finite
group $G$ of real $n\times n$ matrices, the convex hull of its elements in
$\R^{n\times n}$ is a polytope $P(G)$ whose vertices are the group
elements.   A famous
example arises when $G$ is the collection of all $n\times n$
permutation matrices.  In that case, $P(G)$ is the Birkhoff polytope.
Much is known about the Birkhoff polytope (\cite{Brualdi1},
\cite{Brualdi2}, \cite{Brualdi3}, \cite{Brualdi4}) but there are still
open questions
(\cite{Pak}, \cite{Beck}); for instance, its volume is not known in
general.  Our interest in polytopes associated with groups was
inspired by \cite{Billera},  \cite{Brualdi0}, and \cite{Onn}.

In this paper, we consider the case of an important class of
permutation groups, the Frobenius groups.  In sections 1 and 2, we
recall basic facts concerning Frobenius groups and polytopes.  In
section 3, we establish our main result, Theorem~\ref{main-theorem},
identifying the polytope associated with a Frobenius group as a
free sum of simplices.

\section{Frobenius groups}
\begin{dfn} A group $G$ is a {\em Frobenius group} if it has a
proper subgroup $1<H<G$ such that $H\cap(xHx^{-1})=\{1\}$ for all $x\in G\setminus
H$.  The subgroup $H$ is called a {\em Frobenius complement}.
\end{dfn}
We recall some basic facts about Frobenius groups.
Our references are \cite{Alperin}, \cite{Dixon}, and \cite{Huppert}.

 Frobenius groups are precisely those which have representations as
transitive permutation groups which are not regular---meaning there is
at least one non-identity element with a fixed point---and for which
only the identity has more than one fixed point.  In that case, the
stablizer of any point may be taken as a Frobenius complement. On the
other hand, starting with an abstract Frobenius group with complement
$H$, the group $G$ acts on the collection of left-cosets $G/H$ via
left-multiplication. This gives a faithful permutation representation
of $G$ with the desired properties.  The Frobenius complement $H$ is
unique up to conjugation; hence the corresponding permutation
representation is unique up to isomorphism.

A theorem of Frobenius says that if $G$ is a finite Frobenius
group given as a permutation group, as above, the set consisting of
the identity of $G$ and those elements with no fixed points forms a
normal subgroup $N$.  The group $N$ is called the {\em Frobenius
kernel}. We have $G=NH$ with $N\cap H=1$, where $H$ is a Frobenius
complement.  Thus, $G$ is a semi-direct product $N\rtimes H$.
Conversely, if $N$ and $H$ are any two finite groups, and if $\phi$ is
a monomorphism of $H$ into the automorphism group of $N$ for which
each $\phi(h)$ is fixed-point free, then $N\rtimes_\phi H$ is a
Frobenius group with kernel $N$ and complement $H$.  A theorem of J.\
G. Thompson implies that $N$ is nilpotent.

\begin{example}  A few examples of Frobenius groups:
\be
\item The most familiar class of Frobenius groups is the collection of
odd dihedral groups,
\[
D_n=\langle \rho,\phi\mid\rho^n=\phi^2=1,
\rho\phi=\phi\rho^{n-1}\rangle,\ \mbox{$n$ odd},
\]
with Frobenius complement $H=\langle\phi\rangle$ and kernel
$N=\langle\rho\rangle$.  The permutation representation is the usual
group of symmetries of a regular $n$-gon.
\item The alternating group $A_4=\langle(123),(12)(34)\rangle$ is a
Frobenius group with complement $H=\langle(123)\rangle$ and kernel
$N=\langle(12)(34),(13)(24)\rangle$. 
\item Let $p$ and $q$ be prime numbers with
$p\equiv1\mod q$, and let $\phi$ be any monomorphism of $H:=\Z/q\Z$
into the automorphism group (i.e., the group of units) of $N:=\Z/p\Z$.
Then $N\rtimes_\phi H$ is a Frobenius group with complement $H$ and
kernel $N$.  Thus, the unique non-abelian group of size $pq$ is
Frobenius.
\ee  
\end{example}

\section{Polytopes}  Here we recall basic facts we need concerning
polytopes.  Our main reference is~\cite{Ziegler}.  The {\em convex
hull} of a subset $K\subseteq\R^n$ is the intersection of all convex
subsets of $\R^n$ containing $K$.  A {\em polytope} in $\R^n$ is the
convex hull of a finite set of points.  If the polytope $P$ is the
convex hull of points $X=\{p_1,\dots,p_t\}$, then $\dim P$, the {\em
dimension} of $P$, is the dimension of the affine span of $X$,
\[
\textstyle\operatorname{aff}(X):=\{x\in\R^n\mid x=\sum_{i=1}^ta_ip_i,\ a_i\in\R,\
\sum_{i=1}^ta_i=1\}.
\]
An {\em affine relation} on $X$ is an equation
$\sum_{i=1}^ta_ip_i=0$ with $\sum_{i=1}^ta_i=0$. Two such relations are
{\em independent} if their vectors of coefficients are linearly
independent.  If~$q$ is the number of independent affine relations on
$X$, then 
\begin{equation}\label{eq:dim}
\dim P=t-q-1.
\end{equation}
If there are no affine relations, then $P$ is called a {\em
$(t-1)$-simplex}.

A function of the form $A=A(x_1,\dots,x_n)=a_0+\sum_{i=1}^na_ix_i$
with $a_i\in\R$ for all $i$ is called {\em affine}.  The function
$A$ determines two {\em half-spaces}: $A\geq0$ and $A\leq0$.  It is
intuitively obvious, although not trivial to prove, that a set $P$ is
a polytope if and only if it is a compact set which is the
intersection of finitely many half-spaces.

Given a polytope $P\subset\R^n$, we say that $P$ {\em lies on one
side} of the affine function $A$ if $A(p)\geq0$ for all $p\in P$ or if
$A(p)\leq 0$ for all $p\in P$.  In that case, we define a {\em face}
of $P$ as the intersection $P\cap\{p\in\R^n\mid A(p)=0\}$.  The
{\em dimension} of the face is the dimension of its affine span.  The
empty set is the unique face of dimension $-1$.
A vertex is a face of dimension $0$, and a
{\em facet} is a face of dimension $\dim(P)-1$. The collection of
faces of $P$, ordered by inclusion, forms a lattice,
$\mathcal{F}(P)$.  The face lattice is determined by either the facets
or by the vertices in that every face is the intersection of the facets
containing it and is the convex hull of the vertices it contains.
Polytopes $P$ and $Q$ are {\em combinatorially equivalent} if their
face lattices are isomorphic as lattices.  The equivalence class of
$P$ under this relation is the {\em combinatorial type} of $P$.

Polytopes $P\subset\R^n$ and $Q\subset\R^m$ are {\em isomorphic},
denoted $P\approx Q$, if
there is an affine function $A\colon\R^n\to\R^m$, injective when
restriced to the affine span of $P$, such that $A(P)=Q$. Isomorphic
polytopes are combinatorially equivalent. 

We will need the following construction,~(cf.\ \cite{Henk}). Suppose $P$ and $Q$ are
polytopes in $\R^n$ whose relative interiors have nonempty
intersection.  Say
$x\in\operatorname{relint}(P)\cap\operatorname{relint}(Q)$. Further, suppose that the
linear spaces $\operatorname{aff}(P)-x$ and $\operatorname{aff}(Q)-x$ are orthogonal
(hence, $\operatorname{aff}(P)\cap\operatorname{aff}(Q)=\{x\}$).  Define the {\em free
sum}, $P\oplus Q$, to be the convex hull of $P\cup Q$.  The following
isomorphism of lattices is well-known:
\[
\mathcal{F}(P\oplus Q)\approx(\mathcal{F}(P)\times\mathcal{F}(Q))/\sim
\]
where $\sim$ connotes identification of
$(F_1,F_2)\in \mathcal{F}(P)\times\mathcal{F}(Q)$ with $(P,Q)$ if  either
$F_1=P$ or $F_2=Q$.  The lattice structure on the right-hand side has
$(P,Q)$ as the maximal element,  and if $F_1,F_1'$ are faces of $P$ not
equal to $P$ and $F_2,F_2'$ are faces of $Q$ not equal to $Q$, then
$(F_1,F_2)\leq (F_1',F_2')$ if $F_1\subseteq F_1'$ and $F_2\subseteq
F_2'$.  
If $F_1\neq P$ and $F_2\neq Q$, then the face of
$P\oplus Q$ corresponding to $(F_1,F_2)$ is the convex hull of
$F_1\cup F_2$ and has dimension $\dim F_1+\dim F_2+1$.  Otherwise,
$(F_1,F_2)$ corresponds to $P\oplus Q$, itself, which has dimension
$\dim P+\dim Q$.  This construction and identification of
lattices extends in an obvious way to the case of polytopes
$P_1,\dots, P_k$ in $\R^n$ sharing a point $x$ in their relative
interiors and such that their affine spans, when translated by $-x$,
are pairwise orthogonal.
 
A polytope is {\em simplicial} if all of its facets (hence all of its
proper faces) are simplices.  For example, an octahedron is
simplicial.

\section{Frobenius polytopes}
From now on, let $G$ be a finite Frobenius group with kernel $N$ and
complement $H$ acting as a permutation group on the left-cosets $G/H$
via left-multiplication. Our results apply to regular groups as well,
which for convenience we consider to be Frobenius groups with trivial
complement.  In any case, the elements of $N$ serve as a set of
representatives for the distinct cosets of $H$.  We fix a list
$\nu_1=1,\nu_2,\dots,\nu_n$ of the elements of $N$ and define an
action of $G$ on $[n]:=\{1,\dots,n\}$ as follows: for $g\in G$, define
$g(j)=i$ when $i,j\in[n]$ and $g\nu_jH=\nu_iH$.  In this way, we
identify $G$ with a subgroup of the symmetric group, $S_n$, and
identify $H$ with the stabilizer, $G_1$, of $1$ in $G$.

We further identify $G$ with a collection of $n\times n$ permutation
matrices.  The collection of all $n\times n$ real matrices is the
$n^2$-dimensional Euclidean space $\R^{n\times n}$ with coordinates
$\{x_{ij}\}$.  The value of $x_{ij}$ at any matrix $M$ is the $ij$-th
entry of $M$.  For $g\in G$, we take
\[
x_{ij}(g)=\begin{cases}
1&\text{if $g(j)=i$},\\ 
0&\text{if $g(j)\neq i$}.
\end{cases} 
\]
\begin{dfn}
The {\em Frobenius polytope} corresponding to $G$ is the convex hull
of $G\subset\R^{n\times n}$, denoted $P(G)$.
\end{dfn}

\begin{prop}\label{cosets-prop}
Let $G\subset\R^{n\times n}$ be a Frobenius group embedded in
Euclidean space as above.
\be
\item\label{first} $\sum_{g\in hN}g=\mathbf{1}$, where $\mathbf{1}$ is
the $n\times n$ matrix with each entry equal to $1$.
\item\label{second} If $\sum_{g\in G} a_gg=\mathbf{0}\in\R^{n\times n}$
for some $a_g\in\R$, then $a_g=a_{g'}$ for all $g,g'\in hN$.
\ee
\end{prop}
\begin{proof}
We first recall a basic property of Frobenius groups:
\begin{itemize}
\item[($\star$)] for all $i,j\in[n]$, there is precisely one element
$g$ in each coset of $N$ such that $g(j)=i$.
\end{itemize}
To see this, take $H$ as a set of coset representatives of $N$ in $G$
and consider the coset $hN$ with $h\in H$.  Given $i,j\in[n]$, we have
$(\nu_ih\nu_j^{-1})\nu_jH=\nu_iH$.  Since $N$ is
normal, there exists $\nu\in N$ such that
$\nu_ih\nu_j^{-1}=h\nu\in hN$, and $(h\nu)(j)=i$.
Suppose there is also $\nu'\in N$ such that $(h\nu')(j)=i$.  We then
have $h\nu'\nu_jH=h\nu\nu_jH$, whence
$(h\nu\nu_j)^{-1}(h\nu'\nu_j)\in H\cap N=\{1\}$.  Therefore, $\nu=\nu'$,
establishing ($\star$).

Assertion~(\ref{first}) follows immediately from ($\star$).  For each $i,j$ and coset
$hN$, we have $x_{ij}(\sum_{g\in hN}g)=\sharp\{g\in hN\mid
g(j)=i\}=1$. 

Now suppose $\sum_{g\in G}a_gg=0$, as in (\ref{second}).  For each
$i,j\in[n]$, applying the coordinate function $x_{ij}$, it follows
that $\sum_{g\in G: g(j)=i}a_g=0$.  Fix a coset $hN$ and an element
$g'\in hN$. For each $j$ we have $\sum_{g\in
G:g(j)=g'(j)}a_g=0$, and by ($\star$) there is precisely one element $g$
in each coset of $N$ such that $g(j)=g'(j)$.  Further, since no
element besides the identity has more than one fixed point, if
$g(j)=g'(j)$ and $g(j')=g'(j')$, it follows that $g=g'$ or $j=j'$.  Hence,
\begin{eqnarray*}
0&=&\sum_{j=1}^n\sum_{\genfrac{}{}{0pt}{2}{g\in G:}{g(j)=g'(j)}}a_g
=\sum_{j=1}^n\sum_{\genfrac{}{}{0pt}{2}{g\in hN:}{g(j)=g'(j)}}a_g
+\sum_{j=1}^n\sum_{\genfrac{}{}{0pt}{2}{g\in G\setminus hN:}{g(j)=g'(j)}}a_g\\
&=&na_{g'}+\sum_{g\in G\setminus hN}a_g.
\end{eqnarray*}
Solving for $a_{g'}$, we see that its value only depends on
$G\setminus hN$, and (\ref{second}) follows.
\end{proof}

\begin{cor} Let $P(N)$ denote the polytope which is the convex hull of the
Frobenius kernel, $N\subset\R^{n\times n}$.  Then $P(N)$ is a
simplex of dimension $|N|-1$.
\end{cor}
\begin{proof}
The proposition also immediately implies that the elements of $N$ are
affinely independent.
\end{proof}

In the following theorem, for each $h\in H$, let $P(hN)$ denote the
polytope which is the convex hull of the coset $hN\subset\R^{n\times n}$.
Matrix multiplication by $h$ defines a linear automorphism of
$\R^{n\times n}$ which is an isomorphism of $P(N)\approx P(hN)$. By
$P(N)^{\oplus|H|}$, we mean the convex hull of $|H|$ copies of $P(N)$
placed in pairwise orthogonal affine spaces so that the copies of
$P(N)$ meet at their barycenters (vertex average).
\begin{thm}\label{main-theorem}
A Frobenius polytope is a free sum of simplices:
\[
P(G)=\oplus_{h\in H}P(hN)\approx P(N)^{\oplus |H|}.
\]
\end{thm}
\begin{proof}
By Proposition~\ref{cosets-prop}~\eqref{second}, we have $\sum_{g\in
hN}=\mathbf{1}$ for each $h\in H$.  Hence, $\tfrac{1}{|N|}\mathbf{1}$
is in the relative interior of each $P(hN)$.  Translating by this
vector, we must show that
$\{\operatorname{aff}(P(hN))-\tfrac{1}{|N|}\mathbf{1}\}_{h\in H}$
consist of pairwise orthogonal spaces.  To this end, let $h\nu\in hN$
and $h'\nu'\in h'N$ with $h\neq h'$.  We first show that the inner product of these two group elements as points in $\R^{n\times n}$ is $1$.  To say that $\langle h\nu,h'\nu'\rangle=1$ is the same
as saying that $h\nu(j)=h'\nu'(j)$ for precisely one $j$, i.e., that
$\mu:=\nu^{-1}h^{-1}h'\nu'$ has exactly one fixed point.  Since
$\mu\neq 1$ and $G$ is a Frobenius group, the only other possibility
is that $\mu$ has no fixed points and hence is an element of $N$.
However, this would imply that $h^{-1}h'\in H\cap N=\{1\}$ contrary to
the assumption that $h\neq h'$.

Orthogonality quickly
follows: 
\ba 
\langle h\nu-\tfrac{1}{|N|}\mathbf{1},h'\nu'-\tfrac{1}{|N|}\mathbf{1}\rangle
&=& \langle h\nu,h'\nu'\rangle-\langle h\nu,\tfrac{1}{|N|}\mathbf{1}\rangle
-\langle\tfrac{1}{|N|}\mathbf{1},h'\nu'\rangle 
+\langle\tfrac{1}{|N|}\mathbf{1},\tfrac{1}{|N|}\mathbf{1}\rangle\\[5pt] 
&=&1-1-1+1=0.  
\ea 
\end{proof}
We now summarize some immediate consequences of the theorem.
\begin{cor}\label{Frob-poly} Let $|N|=n$ and $|H|=h$.
\be
\item The polytope $P(G)$ is a simplicial polytope of dimension 
$|G|-|H|=(n-1)h$ with $|G|$ vertices and $n^h$ facets.  
\item The faces of $P(G)$ not equal to $P(G)$ itself are exactly the convex
hulls of subsets $X$ of $G$ omitting at least one element from each
coset of $N$. The dimension of the face corresponding to a subset $X$
is $|X|-1$.  
\item  The complement of any set of $h$
elements of $G$, one chosen from each of the cosets of $N$, forms the
set of vertices of a facet, and all facets arise in this way.
\item   
The number of faces of dimension $k$ in $P(G)$ is the coefficient of
$x^{k+1}$ in $x^{(n-1)h+1}+((1+x)^n-x^n)^h$.
\ee
\end{cor}
\begin{remark}
The dimension of $P(G)$ also follows immediately from
Proposition~\ref{cosets-prop}.  It implies that the affine relations
on the elements of $G$ are exactly the affine relations on $|H|$
copies of the matrix $\mathbf{1}$.  There are $|H|-1$ independent such
relations; so, $\dim P(G)=|G|-|H|$ (cf.\ (\ref{eq:dim})).

The fact that each element of $G$ is a vertex of $P(G)$ also follows
from a more general principle.  Multiplication by any element of $G$,
thought of as a permutation matrix, is a linear automorphism of
$\R^{n\times n}$ sending $P(G)$ to itself.  At least one element of
$G$ is a vertex, and since the action of $G$ on itself is transitive,
all elements must be vertices.  Therefore, there are $|G|$ vertices.
\end{remark}
\bibliographystyle{plain}
\bibliography{frob}

\end{document}